\documentclass[11pt,reqno]{amsart}
\usepackage[utf8x]{inputenc}
\usepackage[english]{babel}
\usepackage{amsmath,amsfonts,amssymb,amsthm,array}
\usepackage{graphicx}
\usepackage[colorinlistoftodos]{todonotes}
\usepackage{cases}
\usepackage{tikz}
\usepackage{caption}
\usepackage{subcaption}
\usepackage{cellspace}
\makeatletter
\@namedef{subjclassname@2020}{\textup{2020} Mathematics Subject Classification}
\makeatother

\usepackage[pagebackref=true]{hyperref}

\usepackage{hyperref}
\usepackage{breakurl}   %this allows for url breaks in pdf

\newcommand{\Z}{{\mathbb Z}}
\newcommand{\R}{{\mathbb R}}

\newcommand{\T}{{\mathbb T}}

\newcommand{\area}{{\rm area}}

\newcommand{\sys}{{\rm sys}}
\newcommand{\scg}{{\rm scg}}

\newcommand{\vol}{{\rm vol}}

\newcommand{\length}{{\rm length}}

\numberwithin{equation}{section}

\newtheorem{theorem}{Theorem}[section]
\newtheorem{proposition}[theorem]{Proposition}

\newtheorem{lemma}[theorem]{Lemma}
\newtheorem{claim}[theorem]{Claim}

\theoremstyle{definition}
\newtheorem{definition}[theorem]{Definition}

\newtheorem{remark}[theorem]{Remark}

\long\def\forget#1\forgotten{} 

\begin{document}

\title[Sharp bounds on the length of the shortest closed geodesic]{Sharp upper bounds on the length of the shortest closed geodesic on complete punctured spheres of finite area}

\author[A.~Jabbour]{Antonia Jabbour}
\author[S.~Sabourau]{St\'ephane Sabourau}

\thanks{Partially supported by the ANR project Min-Max (ANR-19-CE40-0014).}

\address{Univ Paris Est Creteil, Univ Gustave Eiffel, CNRS, LAMA, F-94010, Cr\'eteil, France}

\email{antonia.jabbour@u-pec.fr}

\address{Univ Paris Est Creteil, Univ Gustave Eiffel, CNRS, LAMA, F-94010, Cr\'eteil, France}
%\address{\parbox{\linewidth}{Univ Paris Est Creteil, Univ Gustave Eiffel, CNRS, LAMA, F-94010, Cr\'eteil, France \\
%CRM (UMI3457), CNRS, Montr\'eal, QC H3C 3J7, Canada}}

\email{stephane.sabourau@u-pec.fr}



\begin{abstract}
We establish sharp universal upper bounds on the length of the shortest closed geodesic on a punctured sphere with three or four ends endowed with a complete Riemannian metric of finite area.
These sharp curvature-free upper bounds are expressed in terms of the area of the punctured sphere.
In both cases, we describe the extremal metrics, which are modeled on the Calabi-Croke sphere or the tetrahedral sphere.
We also extend these optimal inequalities for reversible and non-necessarily reversible Finsler metrics.
In this setting, we obtain optimal bounds for spheres with a larger number of punctures.
Finally, we present a roughly asymptotically optimal upper bound on the length of the shortest closed geodesic for spheres/surfaces with a large number of punctures in terms of the area.
\end{abstract}

\maketitle

\section{Introduction}

This article deals with universal upper bounds on the length of the shortest closed geodesic on surfaces with a complete Riemannian metric of finite area.
The existence of a closed geodesic on a closed surface with a Riemannian metric follows from a minimization process using Ascoli's theorem in the nonsimply connected case and from Birkhoff's minmax principle in the simply connected case. 
For noncompact surfaces with a complete Riemannian metric, closed geodesics may not exist (the Euclidean plane yields an obvious example).
However, it was proved by Thorbergsson~\cite{thor} for surfaces with at least three ends and by Bangert~\cite{bang} for surfaces with one or two ends that every noncompact surface~$\Sigma$ with a complete Riemannian metric of finite area has a closed geodesic.
This allows us to introduce
\[
\scg(\Sigma) = \inf \{ \length(\gamma) \mid \gamma \mbox{ is a closed geodesic of } \Sigma \}.
\]
Note that in higher dimension the existence of a closed geodesic on a closed Riemannian manifold has been established by Fet and Lyusternik, but whether closed geodesics exist or not on any complete noncompact Riemannian \mbox{$n$-manifold} of finite volume with $n \geq 3$ is an open question; see~\cite[Question~2.3.1]{BM}.

\medskip

In this article, we are interested in finding good (if possible optimal) curvature-free upper bounds on~$\scg(\Sigma)$ for every surface~$\Sigma$ with a complete Riemannian metric of finite area~$A$.
For every nonsimply connected closed surface, it was independently proved by Hebda~\cite{hebda} and Burago-Zalgaller~\cite{BZ} that $\scg(\Sigma) \leq \sqrt{2} \, \sqrt{\area(\Sigma)}$.
In this case, optimal bounds are only known for the torus (Loewner 1949, unpublished, see~\cite{katz}), the projective plane (Pu 1952, \cite{pu}) and the Klein bottle (Bavard 1986, \cite{bav}); see also Section~\ref{sec:proof} and Section~\ref{sec:extremal} for a brief presentation of the extremal inequalities.
For the sphere, it was proved by Croke~\cite{cro88} that $\scg(\Sigma) \leq 31 \, \sqrt{\area(\Sigma)}$.
This bound was subsequently improved in~\cite{NR02}, \cite{sab04} and~\cite{rot}, where the current best bound with $31$ replaced with $4 \sqrt{2}$ is due to Rotman~\cite{rot}.
It is conjectured that the global maximum for the length of the shortest closed geodesic among Riemannian metrics with fixed area on the sphere is attained by the Calabi-Croke sphere (see~\cite{cro88} and~\cite{CK03}), which would yield
\begin{equation} \label{eq:CC}
\scg(S^2) \leq 2^\frac{1}{2} \, 3^\frac{1}{4} \, \sqrt{\area(S^2)}.
\end{equation}
Recall that the Calabi-Croke sphere is defined as the piecewise flat sphere with three conical singularities of angle~$\frac{2\pi}{3}$ obtained by gluing two copies of an equilateral triangle along their boundaries.
Though this conjecture remains wide open, it was proved by Balacheff~\cite{bal10} that the Calabi-Croke sphere is a local maximum (see also \cite{sab10} for an alternate proof extending to the Lipschitz distance topology).
No conjecture is available for other surfaces, except in genus~$3$, where Calabi constructed nonpositively curved piecewise flat metrics with systolically extremal-like properties; see~\cite{cal} and~\cite{sab11} (and~\cite{SY19} for related systolic-like properties in genus~$2$).
Under a nonpositive curvature assumption, extremal systolic inequalities have been established for the genus~$2$ surface and the connected sum of three projective planes; see~\cite{KS06} and~\cite{KS15}.
In both cases, the extremal nonpositively curved metrics are piecewise flat with conical singularities.
It was later proved that this structure is common to all extremal nonpositively curved surfaces; see~\cite{KS}.
Extremal systolic inequalities in a fixed conformal class have been investigated in relation with closed string field theory; see~\cite{HZ20}, \cite{HZ20b}, \cite{NZ19} for the most recent contributions and the references therein.

For noncompact surfaces~$\Sigma$ with a complete Riemannian metric of finite area, it was also shown by Croke~\cite{cro88} that $\scg(\Sigma) \leq 31 \, \sqrt{\area(\Sigma)}$ (without any curvature assumption).
This bound has recently been improved by Beach and Rotman~\cite{BR20}, where the constant~$31$ is replaced with $4\sqrt{2}$ for surfaces with one puncture and with $2\sqrt{2}$ for surfaces with at least two punctures.
The authors also conjectured that the optimal bound for a punctured sphere with at most three punctures is the same as that for the sphere; see~\eqref{eq:CC}.

\medskip

In this article, we show that this conjecture is true for spheres with exactly three punctures and prove an optimal bound for spheres with four punctures.
These are the only new optimal universal upper bounds on the length of the shortest closed geodesic obtained during the almost 35 years since Bavard's inequality on the Klein bottle~\cite{bav}. 
We also improve the best known upper bounds for spheres with a higher number of punctures.
More precisely, we have the following.

%\subsection{Sharp upper bounds on punctured spheres}

\begin{theorem} \label{theo:main}
Let $\Sigma = S^2 \setminus \{ x_1,\dots, x_k \}$ be a $k$-punctured sphere with a complete Riemannian metric of finite area.
Then the following holds.
\begin{enumerate}
\item If $k=3$ then there exists a noncontractible figure-eight geodesic~$\gamma$ on~$\Sigma$ such that
\begin{equation} \label{eq:main1}
\length(\gamma) < 2^{\frac{1}{2}} \cdot 3^{\frac{1}{4}} \, \sqrt{\area(\Sigma)}.
%\area(\Sigma) \geq \frac{\sqrt{3}}{6} \, \scg(\Sigma)^2.
\end{equation}
Furthermore, this inequality is optimal. \label{main1}
\item If $k \geq 4$ then there exists a noncontractible closed geodesic~$\gamma$ on~$\Sigma$ such that
\begin{equation} \label{eq:main2}
\length(\gamma) < 2 \cdot 3^{-\frac{1}{4}} \, \sqrt{\area(\Sigma)}.
%\area(\Sigma) \geq \frac{\sqrt{3}}{4} \, \scg(\Sigma)^2.
\end{equation}
Furthermore, this inequality is optimal when $k=4$. \label{main2}
\forget
\item If $k \geq 5$ then there exists a noncontractible closed geodesic~$\gamma$ on~$\Sigma$ such that
\begin{equation} \label{eq:main3}
\length(\gamma) < 2 \cdot 3^{-\frac{1}{4}} \, \sqrt{\area(\Sigma)} - \varepsilon.
%\area(\Sigma) \geq \frac{\sqrt{3}}{4} \, \scg(\Sigma)^2 + \varepsilon 
\end{equation}
for some universal constant $\varepsilon >0$ (which does not depend on~$k$ or the metric). \label{main3}
Je ne sais pas si c'est vrai???
\forgotten
\end{enumerate}
\end{theorem}

%\noindent Note that $2^{\frac{1}{2}} \cdot 3^{\frac{1}{4}} = 1.861...$ and $2 \cdot 3^{-\frac{1}{4}} = 1.519...$.

%\medskip

\forget
\begin{remark}
A version of this theorem holds true for complete Finsler punctured spheres of finite Holmes-Thompson area; see Section~\ref{sec:finsler} for a brief account on reversible and non-necessarily reversible Finsler metrics and the Holmes-Thompson volume.
In this case, the multiplicative constant $2^{\frac{1}{2}} \cdot 3^{-\frac{1}{4}}$ in~\eqref{eq:main1} is replaced with $2^{-\frac{1}{2}} \cdot 3^{-\frac{1}{2}} \cdot \pi^{\frac{1}{2}}$ for reversible Finsler metrics and with $2^{\frac{1}{2}} \cdot \pi^{\frac{1}{2}}$ for non-necessarily reversible Finsler metrics.
Similarly, the multiplicative constant $2 \cdot 3^{-\frac{1}{4}}$ in~\eqref{eq:main2} and~\eqref{eq:main3} is replaced with $\pi^{\frac{1}{2}}$ for reversible Finsler metrics and with $2 \cdot 3^{-\frac{1}{2}} \cdot \pi^{\frac{1}{2}}$ for non-necessarily reversible Finsler metrics.
Note that the multiplicative constants in the reversible and non-necessarily reversible Finsler cases are also optimal.
\end{remark}
\forgotten

The extremal metric on the three-punctured sphere in~\eqref{eq:main1} is modelled on the Calabi-Croke sphere by attaching three cusps of arbitrarily small area around its singularities.
This can be done keeping the curvature nonpositively curved.

\medskip

The extremal metric on the four-punctured sphere in~\eqref{eq:main2} is modelled on the tetrahedral sphere by attaching four cusps of arbitrarily small area around its singularities.
Here, the tetrahedral sphere is defined as the piecewise flat sphere with four conical singularities of angle~$\pi$ given by the regular tetrahedron.
This can also be done keeping the curvature nonpositively curved.

\medskip

A version of this theorem holds true for complete Finsler punctured spheres of finite Holmes-Thompson area; see Section~\ref{sec:finsler} for a brief account on reversible and non-necessarily reversible Finsler metrics and the Holmes-Thompson volume.

\medskip

For reversible Finsler metrics, we have the following.

\begin{theorem} \label{theo:mainFr}
Let $\Sigma = S^2 \setminus \{ x_1,\dots, x_k \}$ be a $k$-punctured sphere with a complete reversible Finsler metric of finite area.
Then the following holds.
\begin{enumerate}
\item If $k=3$ then there exists a noncontractible figure-eight geodesic~$\gamma$ on~$\Sigma$ such that
\begin{equation} \label{eq:mainFr1}
\length(\gamma) < 2^{-\frac{1}{2}} \cdot 3^{\frac{1}{2}} \cdot \pi^{\frac{1}{2}} \, \sqrt{\area(\Sigma)}.
\end{equation}
%Furthermore, this inequality is optimal. \label{mainFr1}
\item If $k \geq 4$ then there exists a noncontractible closed geodesic~$\gamma$ on~$\Sigma$ such that
\begin{equation} \label{eq:mainFr2}
\length(\gamma) < \pi^{\frac{1}{2}} \, \sqrt{\area(\Sigma)}.
\end{equation}
Furthermore, this inequality is optimal when $k \in \{4,\dots,6\}$. \label{mainFr2}
\end{enumerate}
\end{theorem}

Contrary to \eqref{eq:main1}, the inequality~\eqref{eq:mainFr1} on the three-punctured sphere is not necessarily optimal.
The extremal metric on the four-punctured sphere in~\eqref{eq:mainFr2} is modelled on the sphere~$S^2$ obtained by gluing two copies of the square $[0,1] \times [0,1]$ endowed with the $\ell^1$-metric along their boundary, with four cusps of arbitrarily small area attached around the four vertices of the squares.
Note that the four vertices of the squares and the two centers of the squares are at distance $\frac{1}{2} \,  \scg(S^2)=1$ from each other.
Similarly, the extremal metrics on the five- and six-punctured spheres in~\eqref{eq:mainFr2} are obtained by attaching one or two extra cusps of arbitrarily small area around the centers of the squares.

\medskip

For non-necessarily reversible Finsler metrics, we have the following.

\begin{theorem} \label{theo:mainFnr}
Let $\Sigma = S^2 \setminus \{ x_1,\dots, x_k \}$ be a $k$-punctured sphere with a complete non-necessarily reversible Finsler metric of finite area.
Then the following holds.
\begin{enumerate}
\item If $k=3$ then there exists a noncontractible figure-eight geodesic~$\gamma$ on~$\Sigma$ such that
\begin{equation} \label{eq:mainFnr1}
\length(\gamma) < 2^{\frac{1}{2}} \cdot \pi^{\frac{1}{2}} \, \sqrt{\area(\Sigma)}.
\end{equation}
%Furthermore, this inequality is optimal. \label{mainFnr1}
\item If $k \geq 4$ then there exists a noncontractible closed geodesic~$\gamma$ on~$\Sigma$ such that
\begin{equation*} \label{eq:mainFnr2}
\length(\gamma) < 2 \cdot 3^{-\frac{1}{2}} \cdot \pi^{\frac{1}{2}} \, \sqrt{\area(\Sigma)}.
\end{equation*}
%Furthermore, this inequality is optimal when $k \in \{4,\dots,9\}$. \label{mainFnr2}
\end{enumerate}
\end{theorem}

Contrary to Theorem~\ref{theo:main} and Theorem~\ref{theo:mainFr}, the inequalities in this theorem are not necessarily optimal.
In Section~\ref{sec:extremal}, we present further optimal inequalities on punctured tori, punctured projective planes and punctured Klein bottle with a few ends, where the number of ends (interestingly) depends whether the metric is Riemannian, reversible Finsler or non-reversible Finsler.
These inequalities immediately follow from the corresponding optimal bounds for the underlying closed surfaces.

\medskip

\forget
\begin{theorem} \label{theo:mainF}
Let $\Sigma = S^2 \setminus \{ x_1,\dots, x_k \}$ be a $k$-punctured sphere with a complete Finsler metric of finite area.
Then the following holds.
\begin{enumerate}
\item If $k=3$ then there exists a noncontractible figure-eight geodesic~$\gamma$ on~$\Sigma$ such that
\begin{equation*} \label{eq:mainF1}
\length(\gamma) < 
\begin{cases}
2^{-\frac{1}{2}} \cdot 3^{\frac{1}{2}} \cdot \pi^{\frac{1}{2}} \, \sqrt{\area(\Sigma)} & \mbox{for reversible Finsler metrics} \\
2^{\frac{1}{2}} \cdot \pi^{\frac{1}{2}} \, \sqrt{\area(\Sigma)} & \mbox{for non-reversible Finsler metrics}
\end{cases}
%\area(\Sigma) \geq \frac{\sqrt{3}}{6} \, \scg(\Sigma)^2.
\end{equation*}
Furthermore, this inequality is optimal. \label{mainF1}
\item If $k \in \{4,\dots,8\}$ then there exists a noncontractible closed geodesic~$\gamma$ on~$\Sigma$ such that
\begin{equation*} \label{eq:mainF2}
\length(\gamma) < 
\begin{cases}
\pi^{\frac{1}{2}} \, \sqrt{\area(\Sigma)} & \mbox{for reversible Finsler metrics} \\
2 \cdot 3^{-\frac{1}{2}} \cdot \pi^{\frac{1}{2}} \, \sqrt{\area(\Sigma)} & \mbox{for non-reversible Finsler metrics}
\end{cases}
\end{equation*}
Furthermore, this inequality is optimal. \label{mainF2}
%%%%\forget
\item If $k \geq 9$ then there exists a noncontractible closed geodesic~$\gamma$ on~$\Sigma$ such that
\begin{equation*} \label{eq:mainF3}
\length(\gamma) < 
\begin{cases}
\pi^{\frac{1}{2}} \, \sqrt{\area(\Sigma)} - \varepsilon & \mbox{for reversible Finsler metrics} \\
2 \cdot 3^{-\frac{1}{2}} \cdot \pi^{\frac{1}{2}} \, \sqrt{\area(\Sigma)} - \varepsilon & \mbox{for non-reversible Finsler metrics}
\end{cases}
\end{equation*}
for some universal constant $\varepsilon >0$ (which does not depend on~$k$ or the metric). \label{mainF3}
Je ne sais pas si c'est vrai???
%%%%\forgotten
\end{enumerate}
\end{theorem}
\forgotten

%Note that $2^{-\frac{1}{2}} \cdot 3^{\frac{1}{2}} \cdot \pi^{\frac{1}{2}} = 2.170...$

The following table gives an approximation of the the constant~$c_{\#}$ (optimal in some cases) for the inequality
\[
\scg(\Sigma) \leq c_{\#} \, \sqrt{\area(\Sigma)}
\]
where $\Sigma$ is a $k$-punctured sphere with a complete metric of finite area in the Riemannian, reversible Finsler and non-necessarily reversible Finsler cases when $k=3$ or~$4$.

\begin{table}[htbp!]
\begin{tabular}{|c||SrScSl|SrScSl|SrScSl|}
\hline
$c_{\#}$ & \multicolumn{3}{Sc|}{Riemannian} & \multicolumn{3}{Sc|}{reversible Finsler} & \multicolumn{3}{Sc|}{non-reversible Finsler} \\
\hline
\hline
k=3 & $2^{\frac{1}{2}} \cdot 3^{\frac{1}{4}}$ & $\simeq$ & $1.861...$ & $2^{-\frac{1}{2}} \cdot 3^{\frac{1}{2}} \cdot \pi^{\frac{1}{2}}$ &  $\simeq$ & $2.170...$ & $2^{\frac{1}{2}} \cdot \pi^{\frac{1}{2}}$ &  $\simeq$ & $2.506...$ \\
\hline
k=4 & $2 \cdot 3^{-\frac{1}{4}}$ &  $\simeq$ & $1.519...$ & $\pi^{\frac{1}{2}}$ &  $\simeq$ & $1.772...$ & $2 \cdot 3^{-\frac{1}{2}} \cdot \pi^{\frac{1}{2}}$ &  $\simeq$ & $2.046...$ \\
\hline
\end{tabular}
\end{table}

%\subsection{Spheres with many punctures}

We conclude by proving a roughly asymptotically optimal upper bound on the length of the shortest noncontractible closed geodesic on spheres with a large number of punctures; see Theorem~\ref{theo:g} for a more general statement for genus~$g$ surfaces with $k$ punctures.

\begin{theorem} \label{theo:asymp}
Let $\Sigma = S^2 \setminus \{ x_1,\dots, x_k \}$ be a $k$-punctured sphere with a complete Riemannian metric of finite area, where $k \geq 3$.
Then there exists a noncontractible closed geodesic~$\gamma$ on~$\Sigma$ such that
\begin{equation*} \label{eq:asymp}
\length(\gamma) \leq 4 \sqrt{2} \, \sqrt{\frac{\area(\Sigma)}{k}}.
\end{equation*}
Furthermore, the upper bound is roughly asymptotically optimal in~$k$.
\end{theorem}

Similar upper bounds hold true both in the reversible and non-necessarily reversible Finsler cases (albeit with a different multiplicative constant). 

\medskip

The proofs of our optimal bounds, namely Theorem~\ref{theo:main}, Theorem~\ref{theo:mainFr} and Theorem~\ref{theo:mainFnr}, do not rely on the conformal length method used in \cite{katz}, \cite{pu} and \cite{bav} to establish optimal systolic inequalities.
Instead, we exploit ramified covers from the torus to the sphere (the first one was introduced in~\cite{sab01} and used in~\cite{bal04} and~\cite{sab10} in the same context) to connect the extremal properties of the punctured spheres with the extremal equilateral flat torus in Loewner's systolic inequality; see Theorem~\ref{theo:loewner}. \\

\noindent {\it Acknowledgment.}
The second author would like to thank the Fields Institute and the Department of Mathematics at the University of Toronto, where part of this work was accomplished, for their hospitality.

\section{Finsler metrics and Holmes-Thompson volume} \label{sec:finsler}

This section aims at introducing the notions of Finsler metrics and Holmes-Thompson volume.

\medskip

Let us recall the definition of a Finsler metric.

\begin{definition}
A \emph{Finsler metric} on a manifold~$M$ is a continuous function $F:TM \to [0,\infty)$ on the tangent bundle~$TM$ of~$M$ which is smooth outside the zero section of~$TM$ and whose restriction~$F_x:=F_{|T_x M}$ to each tangent space~$T_xM$ is a (possibly asymmetric) norm, that is, 
\begin{enumerate}
\item Subadditivity: $F_x(u+v) \leq F_x(u) + F_x(v)$ for every $u,v \in T_xM$;
\item Homogeneity: $F_x(tu) = t F_x(u)$ for every $u \in T_xM$;
\item Positive definiteness: $F_x(u) >0$ for every nonzero $u \in T_xM$.
\end{enumerate}
A Finsler metric is \emph{reversible} if $F_x(-u) = F_x(u)$ for every $x \in M$ and $u \in T_xM$.
%In case this condition is not required, we say that the Finsler metric is \emph{non-reversible}.

The length of a piecewise smooth curve $\gamma:[0,1] \to M$ is defined as 
\[
\length(\gamma) = \int_0^1 F(\gamma'(t)) \, dt
\]
and the distance between two points~$x$ and~$y$ in~$M$ is the infimal length of a curve~$\gamma$ in~$M$ joining~$x$ to~$y$.
\end{definition}

We will consider the following notion of volume.

\begin{definition}
The \emph{Holmes-Thompson volume} of an $n$-dimensional Finsler manifold~$M$ is defined as the symplectic volume of its unit co-ball bundle~$B^*M \subseteq T^*M$ divided by the volume~$\epsilon_n$ of the Euclidean unit ball in~$\R^n$.
That is,
\[
\vol(M) = \frac{1}{\epsilon_n} \int_{B^*M} \tfrac{1}{n!} \, \omega_M^n
\]
where $\omega_M$ is the standard symplectic form on~$T^*M$.
\end{definition}

The Holmes-Thompson volume of a Finsler manifold is bounded from above by its Hausdorff measure, with equality if and only if the metric is Riemannian; see~\cite{duran}.
Note also that the Holmes-Thompson volume of a Riemannian manifold agrees with its Riemannian volume.

% \section{Ramified cover from the torus onto the sphere} 
\section{Degree-three ramified cover from the torus onto the Calabi-Croke sphere} \label{sec:degree3}

Consider the piecewise flat sphere $(S^2,g_0)$ with three conical singularities $x_1$, $x_2$, $x_3$ obtained by gluing two copies of a flat unit-side equilateral triangle along their boundaries.
The sphere~$(S^2,g_0)$ is referred to as the Calabi-Croke sphere.

\medskip
	
By the theory of coverings \cite{1}, there exist a degree-three cover \mbox{$\pi_0:\mathbb{T}^2 \to S^2$} ramified over the three vertices $x_1$, $x_2$, $x_3$ of $S^2$, and a deck transformation map $\rho_0:\mathbb{T}^2\to \mathbb{T}^2$ fixing only the ramification points of $\pi_0:\mathbb{T}^2 \to S^2$ with $\rho_0^3= \text{id}_{\mathbb{T}^2}$ and $\pi_0 \circ \rho_0 = \pi_0$.

\medskip

The ramified cover $\pi_0:\mathbb{T}^2 \to S^2$ can also be constructed in a more geometrical way
as follows. First, cut the sphere along the two minimizing arcs of $g_0$ joining
$x_1$ to $x_2$ and $x_1$ to $x_3$. This yields a parallelogram with all sides of unit
length. Then, glue three copies of this parallelogram along the two sides
between $x_3$ and the two copies of $x_1$ to form a hexagon; see Figure~\ref{fig:deg3}.
By identifying the opposite sides of this parallelogram, we obtain an
equilateral flat torus $\mathbb{T}^2$. The isometric rotation, defined on the hexagon,
centered at~$x_3$ and permuting the parallelograms, passes to the quotient and
induces a map $\rho_0:\mathbb{T}^2 \to \mathbb{T}^2$. This map gives rise to a degree-three ramified cover $\pi_0:\mathbb{T}^2 \to S^2$.

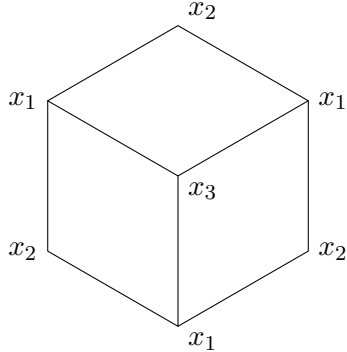
\begin{figure}[htbp!]
\centerline{
\begin{tikzpicture}
\coordinate [label=right:{$x_3$}] (x3) at (0,-0.2);
\coordinate [label=right:{$x_1$}] (x1) at (1.732,1);
\coordinate [label=left:{$x_1$}] (x1) at  (-1.732,1);
\coordinate [label=right:{$x_1$}] (x1) at (0,-2.2);
\coordinate [label=left:{$x_2$}] (x2) at (-1.732,-1);
\coordinate [label=right:{$x_2$}] (x2) at (1.732,-1);
\coordinate [label=right:{$x_2$}] (x2) at (0,2.2);
\draw (0,0)--(1.732,1)--(0,2)--(-1.732,1)--(0,0);
\draw (-1.732,1)--(-1.732,-1)--(0,-2)--(0,0);
\draw (0,-2)--(1.732,-1)--(1.732,1);
\end{tikzpicture}
}
\caption{Degree-three ramified cover of the Calabi-Croke sphere}
\label{fig:deg3}
\end{figure}

Thus, the Calabi-Croke sphere can be described as the quotient of an
equilateral flat torus by the deck transformation map $\rho_0:\mathbb{T}^2 \to \mathbb{T}^2$.

\medskip

Given a Riemannian metric with conical singularities on $S^2$, we will endow~$\mathbb{T}^2$ with the metric pulled back by $\pi_0:\mathbb{T}^2 \to S^2$ and its universal cover~$\mathbb{R}^2$ with the
metric pulled back by the composite map
$$\mathbb{R}^2 \to \mathbb{T}^2 \to S^2.$$

Since the degree of the Riemannian ramified cover $\pi_0:\mathbb{T}^2 \to S^2$ is equal to three, we
have
\begin{equation} \label{eq:3area}
\text{area}(\mathbb{T}^2) = 3 \,\text{area}(S^2).
\end{equation}

\begin{remark} \label{rem:finsler}
The degree-three ramified cover $\pi_0:\mathbb{T}^2 \to S^2$ was first introduced in~\cite{sab01} in relation with extremal properties of the Calabi-Croke sphere regarding the length of the shortest closed geodesic.
It was later used in~\cite{bal10} to show that the Calabi-Croke sphere is a local extremum for the length of the shortest closed geodesic among metrics with fixed area.
A different proof which does not require the uniformization theorem, but still makes use of the degree-three ramified cover $\pi_0:\mathbb{T}^2 \to S^2$, can be found in~\cite{sab10}.
\end{remark}

\section{Degree-two ramified cover from the torus onto the tetrahedral sphere} \label{sec:degree2}

Consider the piecewise flat sphere~$(S^2,g_1)$ with four conical singularities $x_1$, $x_2$, $x_3$, $x_4$ given by the unit-side regular tetrahedron.
The sphere~$(S^2,g_1)$ is referred to as the tetrahedral sphere.

\medskip
		
 By the theory of coverings, there exist a degree-two cover $\pi_1:\mathbb{T}^2 \to S^2$ ramified over the four vertices $x_1$, $x_2$, $x_3$, $x_4$ of $S^2$, and a deck transformation map $\rho_1:\mathbb{T}^2 \to \mathbb{T}^2$ fixing only the ramification points of $\pi_1:\mathbb{T}^2 \to S^2$ with $\rho_1^2= \text{id}_{\mathbb{T}^2}$ and $\pi_1 \circ \rho_1 = \pi_1$.
 
 \medskip
		
The ramified cover $\pi_1:\mathbb{T}^2 \to S^2$ can also be constructed in a more geometrical way as follows. First, cut the sphere along the three minimizing arcs of $g_1$ joining $x_1$ to $x_2$, $x_1$ to $x_3$ and $x_1$ to $x_4$. This yields an equilateral triangle with side length two. Then, glue two copies of this triangle along the side passing through~$x_4$ to form a parallelogram; see Figure~\ref{fig:deg2}.
By identifying the opposite sides of this parallelogram, we obtain an equilateral flat torus $\mathbb{T}^2$. The symmetry, defined on the parallelogram, centered at~$x_4$ and switching the two equilateral triangles, passes to the quotient and induces a map $\rho_1:\mathbb{T}^2 \to \mathbb{T}^2$. This map gives rise to a degree-two ramified cover $\pi_1:\mathbb{T}^2 \to S^2$.

\begin{figure}[htbp!]
\centerline{
		\begin{tikzpicture}
		\coordinate [label=right:{$x_4$}] (x4) at (3.1,1.9);
		\coordinate [label=right:{$x_1$}] (x1) at (6,3.464);
		\coordinate [label=left:{$x_1$}] (x1) at  (0,0);
		\coordinate [label=right:{$x_1$}] (x1) at (4,0);
		\coordinate [label=left:{$x_1$}] (x1) at (2,3.47);
		\coordinate [label=right:{$x_2$}] (x2) at (4,3.64);
		\coordinate [label=right:{$x_2$}] (x2) at (2,-0.2);
		\coordinate [label=right:{$x_3$}] (x3) at (5.1,1.85);
		\coordinate [label=left:{$x_3$}] (x3) at (1,1.85);
		\draw (3,1.732)--(5,1.732)--(6,3.464)--(4,3.464)--(3,1.732);
		\draw (3,1.732)--(2,0)--(4,0)--(5,1.732);
		\draw (2,0)--(0,0)--(1,1.732)--(3,1.732);
		\draw (1,1.732)--(2,3.464)--(4,3.464);
		\draw [dashed](4,3.464)--(5,1.732);
		\draw [dashed](2,3.464)--(4,0);
		\draw [dashed](1,1.732)--(2,0);
		\end{tikzpicture}
}
\caption{Degree-two ramified cover of the tetrahedral sphere}
\label{fig:deg2}
\end{figure}
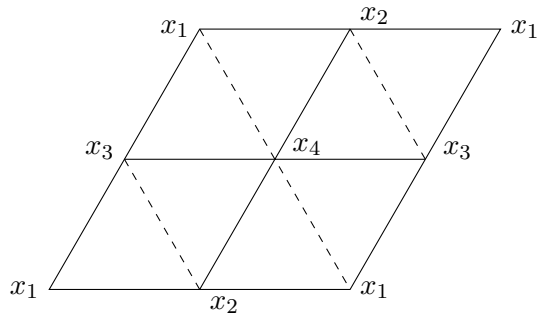
	
Thus, the tetrahedral sphere~$(S^2,g_1)$ can be described as the quotient of an equilateral flat torus by the deck transformation map $\rho_1:\mathbb{T}^2 \to \mathbb{T}^2$.

\medskip
		
Given a Riemannian metric with conical singularities on $S^2$, we will endow~$\mathbb{T}^2$ with the metric pulled back by $\pi_1: \mathbb{T}^2 \to S^2$ and its universal cover~$\mathbb{R}^2$ with the metric pulled back by the composite map
$$\mathbb{R}^2 \to \mathbb{T}^2 \to S^2.$$
		
Since the degree of the Riemannian ramified cover $\pi_1: \mathbb{T}^2 \to S^2$ is equal to two, we have
\begin{equation} \label{eq:2area}
\text{area}(\mathbb{T}^2) = 2 \, \text{area}(S^2).
\end{equation}

\section{Proof of the main theorem} \label{sec:proof}

In this section, we recall some basic results in systolic geometry and prove the main theorem of this article, both in the Riemannian case and in the Finsler case.

\begin{definition}
Let $M$ be a surface with a complete (Riemannian or Finsler) metric.
The \emph{systole} of~$M$ is defined as
\[
\sys(M) = \inf \{ \length(\gamma) \mid \gamma \mbox{ is a noncontractible loop of } M \}.
\]
When $M$ is closed, the systole is attained by the length of a noncontractible closed geodesic referred to as a \emph{systolic loop} of~$M$.
\end{definition}

We will also need the following extension of the notion of systole.

\begin{definition} \label{def:marked}
Let $M$ be a surface with $k$ punctures and $p$ marked points $x_1, \dots, x_p$, endowed with a complete (Riemannian or Finsler) metric.
A loop of~$M$ is \emph{admissible} if it lies in $M' = M \setminus \{ x_1, \dots, x_p \}$ and is not homotopic in~$M'$ to a point, a multiple of a noncontractible simple loop of a cusp, or a multiple of some small circle around a marked point~$x_i$. 
The \emph{marked homotopy systole} of~$M$ is the infimal length of the admissible loops of~$M$.
It is denoted by~$\sys_*(M)$.
\end{definition}

Let us recall Loewner's systolic inequality in the Riemannian case (unpublished), see~\cite{katz}, in the reversible Finsler case, see~\cite{sab10}, and in the non-necessarily reversible case, see~\cite{ABT}.

\begin{theorem}[\cite{katz}, \cite{sab10}, \cite{ABT}] \label{theo:loewner}
Let $\T^2$ be a torus.
Then the following statements hold true.
\begin{enumerate}
\item For every Riemannian metric on~$\T^2$, %there exists a noncontractible closed geodesic~$\gamma$ of~$\T^2$ such that
\begin{equation} \label{eq:loewner1}
\sys(\T^2) \leq 2^{\frac{1}{2}} \cdot 3^{-\frac{1}{4}} \, \sqrt{\area(\T^2)}
\end{equation}
with equality if and only if $\T^2$ is an equilateral flat torus. \label{loewner1}
\item For every reversible Finsler metric on~$\T^2$, %there exists a noncontractible closed geodesic~$\gamma$ of~$\T^2$ such that
\begin{equation} \label{eq:loewner2}
\sys(\T^2) \leq 2^{-\frac{1}{2}} \cdot \pi^{\frac{1}{2}} \, \sqrt{\area(\T^2)}
\end{equation}
with equality if $\T^2$ is a square flat torus endowed with the $\ell^1$- or $\ell^\infty$-metric. \label{loewner2}
\item For every non-necessarily reversible Finsler metric on~$\T^2$, %there exists a noncontractible closed geodesic~$\gamma$ of~$\T^2$ such that
\begin{equation} \label{eq:loewner3}
\sys(\T^2) \leq 2^{\frac{1}{2}} \cdot 3^{-\frac{1}{2}} \cdot \pi^{\frac{1}{2}} \, \sqrt{\area(\T^2)}
\end{equation}
with equality if $\T^2$ is homothetic to the quotient of~$\R^2$, endowed with the non-symmetric norm whose unit disk is the triangle with vertices $(1,0)$, $(0,1)$ and $(-1,-1)$, by the lattice~$\Z^2$. \label{loewner3}
\end{enumerate}
\end{theorem}

We can now proceed to the proof of the main theorem.

\begin{proof}[Proof of Theorem~\ref{theo:main}]
Consider the case~\eqref{main1}.
Let $\Sigma$ be a $3$-punctured sphere with a complete Riemannian metric of finite area.
%Fix~$\varepsilon>0$.
Take three cylindrical ends~$\bar{C}_1$, $\bar{C}_2$, $\bar{C}_3$ of~$\Sigma$.  %with $\area(\bar{C}_i) < \varepsilon$.
For every $i=1,\dots,3$, take a cylindrical end $C_i \subseteq \bar{C}_i$ with 
\begin{equation} \label{eq:distC_i}
d(C_i,\partial \bar{C}_i) > 2^{\frac{1}{2}} \cdot 3^{\frac{1}{4}} \, \sqrt{\area(\Sigma)}.
\end{equation}
Collapse every end~$C_i$ to a point~$x_i$.
This gives rise to a sphere~$S^2$ with a Riemannian metric with three singularities~$x_1$, $x_2$, $x_3$.
Note that $\area(S^2) < \area(\Sigma)$.

\medskip

Consider the degree-three ramified cover $\pi_0:\T^2 \to S^2$ with branched points~$x_1$, $x_2$, $x_3$ described in Section~\ref{sec:degree3}.
Denote by~$p_i$ the preimage of~$x_i$ under~$\pi_0:\T^2 \to S^2$ for $i=1,\dots,3$.
Endow~$\T^2$ with the singular pullback Riemannian metric.
The metric on~$\T^2$ can be smoothed out in the neighborhood of its singularities, keeping the area and the systole fixed.
By Loewner's inequality~\eqref{eq:loewner1} and the relation~\eqref{eq:3area}, there exists a noncontractible closed geodesic~$\gamma$ on~$\T^2$ with
\begin{equation} \label{eq:gamma<S}
\length(\gamma) \leq 2^{\frac{1}{2}} \cdot 3^{\frac{1}{4}} \, \sqrt{\area(S^2)}.
\end{equation}
The loop~$\gamma$ does not pass through any singularity~$p_i$.
Otherwise, by the length upper bound~\eqref{eq:gamma<S} and the distance lower bound~\eqref{eq:distC_i}, it would lie in the topological disk of~$\mathbb{T}^2$ given by the quotient~$\bar{C}_i/C_i$.
This would contradict the noncontractibility of~$\gamma$ in~$\mathbb{T}^2$.
Thus, the systolic loop~$\gamma$ projects to a closed geodesic of~$S^2 \setminus \{x_1,x_2,x_3\} \subseteq \Sigma$ under $\pi_0:\T^2 \to S^2$.
By~\cite[Lemma~7.1]{sab10}, this closed geodesic is a figure-eight geodesic of~$S^2$ (with exactly one singularity~$x_i$ in each of the three domains of~$S^2$ it bounds).
This concludes the proof of the case~\eqref{main1} in Theorem~\ref{theo:main}.

\medskip

In the case~\eqref{main2}, the proof is similar.
Start with a $4$-punctured sphere~$\Sigma$ with a complete Riemannian metric of finite area.
Take four cylindrical ends $C_1$, $C_2$, $C_3$, $C_4$ of~$\Sigma$ of small area located far away from the core of the surface.
Collapse the cylindrical ends into points~$x_i$.
This gives rise to a sphere~$S^2$ with a Riemannian metric with four singularities~$x_i$.
Consider the degree-two ramified cover $\pi_1:\T^2 \to S^2$ with branched points~$x_i$ described in Section~\ref{sec:degree2}.
By Loewner's inequality~\eqref{eq:loewner1} and the relation~\eqref{eq:2area}, there exists a noncontractible closed geodesic~$\gamma$ on~$\T^2$ with
\[
\length(\gamma) \leq 2 \cdot 3^{-\frac{1}{4}} \, \sqrt{\area(S^2)}.
\]
As previously, the systolic loop~$\gamma$ does not pass through a ramification point of $\pi_1:\T^2 \to S^2$ and projects to a closed geodesic\footnote{Arguing as in~\cite[Lemma~7.1]{sab10}, one can show that the systolic loop~$\gamma$ of~$\T^2$ projects either to a simple closed geodesic surrounding exactly two branched points of~$S^2$ on each side, or to a figure-eight geodesic with exactly one or two branched points in each of the three domains of~$S^2$ it bounds.} of \mbox{$S^2 \setminus \{x_1,x_2,x_3,x_4 \} \subseteq \Sigma$}.
This concludes the proof of the case~\eqref{main2} in Theorem~\ref{theo:main}.

In both cases, the extremal metrics are described in the introduction right after Theorem~\ref{theo:main}.
\end{proof}

\begin{remark} 
In the Finsler case, we simply need to replace Loewner's inequality~\eqref{eq:loewner1} with~\eqref{eq:loewner2} for reversible Finsler metrics, and with~\eqref{eq:loewner3} for non-necessarily reversible Finsler metrics.
This leads to the Finsler version of the main theorem given by Theorem~\ref{theo:mainFr} and Theorem~\ref{theo:mainFnr}.
The only minor novelty is when $k=5$ or~$6$.
In this case, we take $k$ cylindrical ends~$C_i$ of~$\Sigma$ of small area located far away from the core of the surface, and collapse the cylindrical ends into points~$x_i$.
Consider the degree-two ramified cover $\pi_1:\T^2 \to S^2$ branched \emph{only} at four points $x_1,\ldots,x_4$ as previously.
Apply the Finsler version of Loewner's inequality and observe that the systolic loops of~$\T^2$ do not pass through the preimages~$\pi_1^{-1}(x_i)$ of the singularities of~$S^2$ and project to closed geodesics of $S^2 \setminus \{x_1,\ldots,x_k \} \subseteq \Sigma$ as required.
\end{remark}

\begin{remark} 
Contrary to the Riemannian case, the extremal (reversible or non-reversible) Finsler metric on~$\T^2$ does not pass to the quotient under the deck transformation groups of $\pi_0:\T^2 \to S^2$, which explains why the inequalities~\eqref{eq:mainFr1} and~\eqref{eq:mainFnr1} may not be optimal.
The same occurs for the extremal non-reversible Finsler metric on~$\T^2$ with the deck transformation group of $\pi_1:\T^2 \to S^2$.
However, the extremal reversible Finsler metric on~$\T^2$ does pass to the quotient under the deck transformation group of $\pi_1:\T^2 \to S^2$.
In this case, the inequality~\eqref{eq:mainFr2} is optimal and the approximating metrics are described in the introduction right after Theorem~\ref{theo:mainFr}.
\end{remark}

\begin{remark} \label{rem:disj}
One may wonder if our technique can be applied to other ramified covers $\T^2 \to S^2$ in order to derive sharp upper bounds on the length of the shortest closed geodesics on other Riemannian punctured spheres~$\Sigma$.
At the heart of the matter is the property that the extremal equilateral flat metric on~$\T^2$ should induce an extremal Riemannian metric on~$\Sigma$ but also on~$S^2$ with marked points/branched points~$x_i$ corresponding to the ends of~$\Sigma$.
In particular, the marked homotopy systole of~$S^2$ should be greater or equal to the systole of~$\T^2$.
This implies that the ramification points~$p_i$ of~$\T^2$ must be at distance at least $\tfrac{1}{2} \, \sys(\T^2)$ from each other.
Thus, the open disks $D(p_i,\frac{1}{4} \, \sys(\T^2))$ must be disjoint.
Since the area of each of these flat disks is equal to $\tfrac{\pi}{16} \, \sys(\T^2)^2$, we deduce that the number of ramification points of~$\T^2$ does not exceed 
\[
\frac{\area(\T^2)}{\tfrac{\pi}{16} \, \sys(\T^2)^2} = \tfrac{8}{\pi} \sqrt{3} = 4.4...
\]
Therefore, the number of ramification points is at most~$4$.
In conclusion, our method to find extremal Riemannian metrics based on Loewner's inequality on the torus cannot apply to punctured spheres with more than $4$ ends.
\end{remark}

\section{Extremal metrics on noncompact surfaces} \label{sec:extremal}

In this section, we present other examples of noncompact surfaces admitting sharp upper bounds on the length of their shortest closed geodesic.

\begin{proposition}
Let $M$ be a closed surface with a systolically extremal (Riemannian or Finsler) metric.
Denote by~$\Sigma = M \setminus \{x_1,\dots,x_k\}$ the surface~$M$ with $k$ punctures.
Then every complete (Riemannian or Finsler) metric on~$\Sigma$ satisfies
\begin{equation} \label{eq:c}
\sys_*(\Sigma) \leq c(M) \, \sqrt{\area(\Sigma)}
\end{equation}
where $c(M) = \frac{\sys(M)}{\sqrt{\area(M)}}$.
\end{proposition}

\begin{proof}
To prove this upper bound on $\sys_*(\Sigma)$, simply collapse small enough and far enough cylindrical ends~$C_i$ of~$\Sigma$.
The resulting surface~$M'$ (where the metric is smoothed out) is homeomophic to~$M$ and satisfies $\sys_*(\Sigma) \leq \sys(M')$ and $\area(M') \leq \area(\Sigma)$.
Since the metric on~$M$ is systolically extremal, we clearly have $c(M') \leq c(M)$ and the desired result immediately follows.
\end{proof}

The inequality~\eqref{eq:c} is not optimal when $k$ is large, see Theorem~\ref{theo:g}, but it is for small values of~$k$.
For this, one needs to find $k$ points on~$M$ at distance at least~$\frac{1}{2} \, \sys(M)$ from each other.
Compare with Remark~\ref{rem:disj}.

For instance, we can consider the extremal (Riemannian or Finsler) metrics on the torus as follows; see Theorem~\ref{theo:loewner}.
The equilateral flat torus (see Theorem~\ref{theo:loewner}.\eqref{loewner1}) admits $4$ such points; see Figure~\ref{fig:Riemannian}.
Attaching cusps of arbitrarily small area around these $4$ points, we construct an almost extremal Riemannian metric on the torus with $k$ punctures, where $k \leq 4$.
Similarly, the square flat torus with the $\ell^1$-metric (see Theorem~\ref{theo:loewner}.\eqref{loewner2}) admits $8$ such points; see Figure~\ref{fig:reversible}.
As previously, we can construct an almost extremal reversible Finsler metric on the torus with $k$ punctures, where $k \leq 8$.
Finally, the square torus with the extremal non-reversible Finsler metric (see Theorem~\ref{theo:loewner}.\eqref{loewner3}) admits $9$ points whose distance, back and forth, between any pair of them is at least~$\sys(\T^2)$; see Figure~\ref{fig:non-reversible}. (Note that the asymmetric distance between two of these points might be less than $\frac{1}{2} \, \sys(\T^2)$ but the distance in the opposite direction makes up for it and their sum is at least~$\sys(\T^2)$.)
As previously, we can construct an almost extremal non-reversible Finsler metric on the torus with $k$ punctures, where $k \leq 9$.

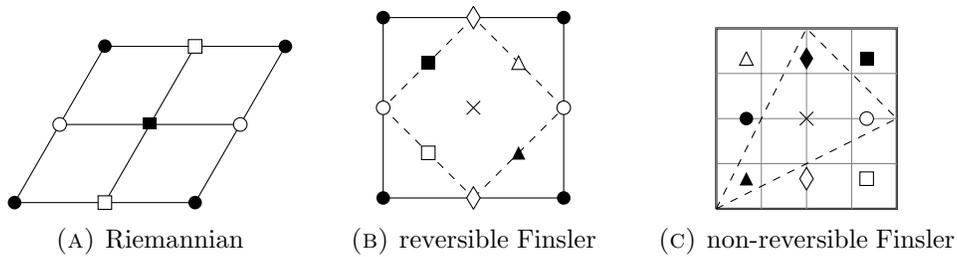
\begin{figure}[htbp!]
		\centering
		\begin{subfigure}[b]{0.3\textwidth}
			\centering
				\begin{tikzpicture}[scale=0.6]
			\draw (3,1.732)--(5,1.732)--(6,3.464)--(4,3.464)--(3,1.732);
			\draw (3,1.732)--(2,0)--(4,0)--(5,1.732);
			\draw (2,0)--(0,0)--(1,1.732)--(3,1.732);
			\draw (1,1.732)--(2,3.464)--(4,3.464);
			
			\fill (0,0) circle (0.15cm);
			\fill (6,3.464) circle (0.15cm);
			\fill (4,0) circle (0.15cm);
			\fill (2,3.47) circle (0.15cm);
			\filldraw [fill=white,draw=black] (5,1.732) circle (0.15cm);
			\filldraw [fill=white,draw=black] (1,1.732) circle(0.15cm);
			\fill (2.85,1.622) rectangle (3.15,1.882); 
			\filldraw [fill=white,draw=black] (1.85,-0.15) rectangle (2.15,0.15);
			\filldraw [fill=white,draw=black] (4.15,3.614) rectangle (3.85,3.314);
			\end{tikzpicture}
			\caption{Riemannian}
		\label{fig:Riemannian}
		\end{subfigure}
		\hfill
		\begin{subfigure}[b]{0.3\textwidth}
			\centering
		  	\begin{tikzpicture}[scale=0.6]
		
		\draw (-2,-2)--(2,-2)--(2,2)--(-2,2)--(-2,-2);
		\draw [dashed] (0,-2)--(2,0)--(0,2)--(-2,0)--(0,-2);
		\draw (-0.15,-0.15)--(0.15,0.15);
		\draw (-0.15,0.15)--(0.15,-0.15);
		\fill (2,2) circle (0.15cm);
		\fill (2,-2) circle (0.15cm);
		\fill (-2,-2) circle (0.15cm);
		\fill (-2,2) circle (0.15cm);
		\draw [fill=white,draw=black](2,0) circle (0.15cm);
		\filldraw [fill=white,draw=black] (-2,0) circle (0.15cm);
		
		\fill (-1.15,0.85) rectangle (-0.85,1.15);
		\filldraw [fill=white,draw=black](1.15,0.85)--(1,1.15)--(0.85,0.85)--(1.15,0.85);
		
		\filldraw [fill=white,draw=black] (-1.15,-1.15) rectangle(-0.85,-0.85);
		\fill(1.15,-1.15)--(1,-0.85)--(0.85,-1.15)--(1.15,-1.15);
		\filldraw [fill=white,draw=black] (0,2.25)--(-0.15,2)--(0,1.75)--(0.15,2)--(0,2.25);
		\filldraw [fill=white,draw=black] (0,-2.25)--(0.15,-2)--(0,-1.75)--(-0.15,-2)--(0,-2.25);
		\end{tikzpicture}
			\caption{reversible Finsler}
	\label{fig:reversible}
		\end{subfigure}
		\hfill
		\begin{subfigure}[b]{0.32\textwidth}
			\centering
			\begin{tikzpicture}[scale=0.6]
		
		\draw[thick] (-2,-2)--(2,-2)--(2,2)--(-2,2)--(-2,-2);
		\draw[step=1cm,gray,very thin] (-2,-2) grid (2,2);
		\draw [dashed] (-2,-2)--(2,0)--(0,2)--(-2,-2);
		\fill (-1.333,0) circle (0.15cm);
		\filldraw[fill=white,draw=black] (1.333,0) circle (0.15cm);
		\draw (-0.15,-0.15)--(0.15,0.15);
		\draw (-0.15,0.15)--(0.15,-0.15);
		\draw (1.183,-1.483) rectangle (1.483,-1.183);
		\fill (1.183,1.183) rectangle (1.483,1.483);
		\draw (-1.333,1.483)--(-1.183,1.183)--(-1.483,1.183)--(-1.333,1.483);
		\fill (-1.333,-1.183)--(-1.183,-1.483)--(-1.483,-1.483)--(-1.333,-1.183);
		\fill (0,1.583)--(-0.15,1.333)--(0,1.083)--(0.15,1.333)--(0,1.583);
		\filldraw [fill=white, draw=black] (0,-1.083)--(-0.15,-1.333)--(0,-1.583)--(0.15,-1.333)--(0,-1.083);
		\end{tikzpicture}
		\caption{non-reversible Finsler}
		\label{fig:non-reversible}
		\end{subfigure}
		\caption{Separated points on the torus}
		\label{fig:torus}
\end{figure}

The same construction applies to the projective plane where the extremal metric is given by the canonical metric both in the Riemannian and Finsler cases; see~\cite{pu} and~\cite{iva}.
More precisely, we can construct an almost extremal metric on the projective plane with $k$ punctures, where $k \leq 3$.

This construction also applies to the Klein bottle where the extremal metric is known both in the Riemannian and reversible Finsler settings.
Specifically, the extremal Riemannian Klein bottle is obtained by attaching along their boundary two copies of the Mobius band defined as the quotient of the $\frac{\pi}{4}$-neighborhood of the equator on the standard sphere by the antipodal map; see~\cite{bav}.
While the extremal Finsler Klein bottle is the square flat Klein bottle with the $\ell^1$-metric; see~\cite{cos}.
Thus, we can construct an almost extremal metric on the Klein bottle with $k$ punctures, where $k \leq 4$ in the Riemannian case and $k \leq 8$ in the reversible Finsler case.

\forget
\begin{remark}
The systolic inequalities of Theorem~\ref{theo:loewner} hold true for every torus~$\Sigma$ with $k$ punctures with a complete (Riemannian or Finsler) metric of finite area by replacing $\sys(\Sigma)$ with $\sys_*(\Sigma)$.
Note that this inequality is sharp if $k \leq 4$.
An almost extremal metric is given by the flat equilateral torus~$\T^2$ with $k$ arbitrarily small cusps attached around $k$ points at distance $\tfrac{1}{2} \, \sys(\T^2)$ from each other, where $k \leq 4$.
To prove this sharp upper bound on $\sys_*(\Sigma)$, simply collapse small enough and far enough cylindrical ends~$C_i$ of~$\Sigma$ and apply Theorem~\ref{theo:loewner} to the resulting torus~$\T^2$ (where the metric is smoothed out).
Since $\sys_*(\Sigma) \leq \sys(\T^2)$ and $\area(\T^2) \leq \area(\Sigma)$, we immediately obtain the desired result.
\end{remark}
\forgotten

\section{Surfaces with many punctures}

In this section, we show a roughly asymptotically optimal upper bound on the length of the shortest closed geodesic on a surface with a large number of punctures.

\medskip

We will need the following result, which can be found in~\cite[Lemma~6.5]{BPS}.

\begin{lemma} \label{lem:BPS}
Let $M$ be a closed surface with a Riemannian metric and $k$ marked points $x_1, \dots, x_k$, with $k \geq 3$.
Fix $R \in (0,\tfrac{1}{4} \, \sys_*(M)]$.
Then there exists a closed Riemannian surface~$\bar{M}$ such that
\begin{align}
\area(\bar{M}) & \leq \area(M) \label{BPS1} \\
\sys_*(\bar{M}) & = \sys_*(M) \label{BPS2} \\
\area \, \bar{D}(R) & \geq \tfrac{1}{2} R^2 \label{BPS3}
\end{align}
for every disk~$\bar{D}(R)$ of radius~$R$ in~$\bar{M}$.
%\begin{enumerate}
%\item $\area(\bar{M}) \leq \area(M)$; \label{BPS1}
%\item $\sys_*(\bar{M}) = \sys_*(M)$; \label{BPS2}
%\item the area of every disk of radius~$R \in (0,\tfrac{1}{4} \, \sys_*(M)]$ in~$\bar{M}$ is at least~$\tfrac{1}{2} R^2$. \label{BPS3}
%\end{enumerate}
\end{lemma}

%\begin{definition}
%A surface of signature~$(g,k)$ is a surface of genus~$g$ with $k$ punctures.
%\end{definition}

The following result implies Theorem~\ref{theo:asymp} when $g=0$.

\begin{theorem} \label{theo:g}
Let $\Sigma$ be a surface of genus~$g$ with $k$ punctures, endowed with a complete Riemannian metric of finite area.
Then
\[
\sys_*(\Sigma) \leq C \frac{\log(g+2)}{\sqrt{g+k+1}} \, \sqrt{\area(\Sigma)}
\]
where $C$ is an explicit universal constant.
\end{theorem}

\begin{proof}%[Proof of Theorem~\ref{theo:asymp}]
Take $k$ cylindrical ends~$C_i \subseteq \Sigma$ far away from the core of~$\Sigma$ so that
\begin{equation} \label{eq:CiCj}
d(C_i,C_j) > \sys_*(\Sigma)
\end{equation}
for every $i \neq j$, and
\begin{equation} \label{eq:Ci}
\length(\alpha) > \sys_*(\Sigma)
\end{equation}
for every arc $\alpha$ of~$\Sigma$ with endpoints in~$C_i$ inducing a nontrivial class in~$\pi_1(\Sigma,C_i)$.

Collapse every end~$C_i$ to a point~$x_i$.
Denote by~$M$ the resulting closed surface with $k$ marked points $x_1,\dots,x_k$.
The Riemannian metric on~$\Sigma$ induces a metric on~$M$ that can be smoothed out in the neighborhood of the singularities~$x_i$, keeping the area and the marked homotopy systole fixed.
Note that $\area(M) \leq \area(\Sigma)$.

\begin{claim}
We have
\[
\sys_*(\Sigma) \leq \sys(M).
\]
\end{claim}

\begin{proof}
Let us show that $\length(\gamma) \geq \sys_*(\Sigma)$ for every noncontractible loop~$\gamma$ of~$M$.
By~\eqref{eq:CiCj}, we can assume that the loop~$\gamma$ passes through at most one singularity of~$M$, otherwise we are done.
We can further assume that the loop~$\gamma$ does not pass through any singularity~$x_i$ of~$M$.
Otherwise, it would admit an arc~$\alpha \subseteq \Sigma$ with endpoints in~$C_i$ inducing a nontrivial class in~$\pi_1(\Sigma,C_i)$ as a lift under the quotient map $\Sigma \to M$.
By~\eqref{eq:Ci}, we would be done.
Thus, the loop~$\gamma$ of~$M$ also lies in~$\Sigma$.
Furthermore, the loop~$\gamma$ is noncontractible in~$\Sigma$, even after collapsing the ends of~$\Sigma$.
It follows that $\gamma$ is an admissible loop of~$\Sigma$.
Therefore, $\length(\gamma) \geq \sys_*(\Sigma)$.
\end{proof}

The roughly asymptotically optimal systolic inequality for closed surfaces of large genus \cite{gro83} \cite{gro96} (see also \cite{bal04} and \cite{KS05} for alternate proofs) applied to~$M$, combined with the relations $\sys_*(\Sigma) \leq \sys(M)$ and $\area(M) \leq \area(\Sigma)$, shows that
\begin{equation} \label{eq:g}
\sys_*(\Sigma) \leq C' \, \frac{\log(g+2)}{\sqrt{g+1}} \, \sqrt{\area(\Sigma)}
\end{equation}
for some explicit universal constant~$C'$.
This proves the theorem when $k=0$.

Now, consider the closed surface~$\bar{M}$ obtained by applying Lemma~\ref{lem:BPS} to the closed surface~$M$ with its $k$ marked points, with $R=\tfrac{1}{4} \, \sys_*(M)$.
Observe that 
\[
d_{\bar{M}}(x_i,x_j) \geq \tfrac{1}{2} \, \sys_*(M).
\]
Otherwise we could find a figure-eight curve on~$\bar{M}$ of length less than $\sys_*(M)$, in the neighborhood of the segment~$[x_i,x_j]$, surrounding both $x_i$ and~$x_j$.
This would contradict the relation~\eqref{BPS2}.

It follows that the open disks $\bar{D}(x_i, \tfrac{1}{4} \, \sys_*(M))$ of~$\bar{M}$ are disjoint.
Combined with~\eqref{BPS1}, we derive 
\[
\area(\Sigma) \geq \area(\bar{M}) \geq \sum_{i=1}^k \area \, \bar{D}(x_i, \tfrac{1}{4} \, \sys_*(M)).
\]

Now, by~\eqref{BPS3}, we have
\[
\area \, \bar{D}(x_i, \tfrac{1}{4} \, \sys_*(M)) \geq \tfrac{1}{32} \, \sys_*(M)^2.
\]
Hence,
\begin{equation} \label{eq:k}
\sys_*(\Sigma) \leq \frac{4 \sqrt{2}}{\sqrt{k}} \, \sqrt{\area(\Sigma)}.
\end{equation}

Now, if $\displaystyle k \geq 10 \, \frac{g+1}{\log(g+2)^2}$, then $\displaystyle k \geq \frac{1}{10} \, \frac{g+k+1}{\log(g+2)^2}$ and the desired upper bound on~$\sys_*(\Sigma)$ follows from~\eqref{eq:k}.
Otherwise, if $\displaystyle k \leq 10 \, \frac{g+1}{\log(g+2)^2}$, then $\sqrt{g+1} \geq \frac{1}{10} \, \sqrt{g+k+1}$ and the desired upper bound follows from~\eqref{eq:g}.
%Combining the inequalities~\eqref{eq:g} and~\eqref{eq:k}, we derive the desired result.
\end{proof}

\begin{remark}
Theorem~\ref{theo:g} extends to Finsler metrics.
Indeed, given a non-necessarily reversible Finsler metric~$F$ on~$\Sigma$, we can replace~$F$ with a reversible Finsler metric~$F'$ defined by $F'(v)=F(v)+F(-v)$.
Then we replace~$F'$ with the continuous Riemannian metric~$g$ whose unit disk agrees with the inner Loewner ellipsoid associated to the unit tangent disk of~$F'$.
By construction, $\sys_*(\Sigma,F) \leq \sys_*(\Sigma,g)$ and $\area(\Sigma,g) \leq \lambda \, \area(\Sigma,F)$ for some explicit universal constant~$\lambda$.
(We refer to the proofs of Corollary~4.12 and Theorem~4.13 in~\cite{ABT} for the details.)
Thus, we can apply Theorem~\ref{theo:g} to the Riemannian metric~$g$ and immediately derive a similar upper bound on the length of the shortest closed geodesic of~$F$ in terms of the Holmes-Thompson area of~$F$ (with a different multiplicative constant).
\end{remark}

\end{document}